\documentclass{article}
\usepackage{graphicx} 
\usepackage{amsfonts, amsmath, amssymb, amsthm}
\usepackage{mathrsfs}
\usepackage{hyperref}
\usepackage{xcolor} 

\newcommand{\R}{\mathbb{R}}

\newtheorem{proposition}{Proposition}
\newtheorem{theorem}{Theorem}
\newtheorem{remark}{Remark}

\newtheorem{corollary}{Corollary}

\title{An implicit two-phase obstacle-type problem for the $p$-Laplacian with the fractional gradient}
\author{Pedro Miguel Campos\and José Francisco Rodrigues}
\date{\today}
\begin{document}
\maketitle

\begin{center}
    \textit{Dedicated to Nina Nikolaevna Ural'tseva on the occasion of her 90th birthday}
\end{center}

\begin{abstract}
    In this work we study an inhomogeneous two-phase obstacle-type problem associated to the $s$-fractional $p$-Laplacian. Besides the existence and uniqueness of solutions, we study the convergence of the solutions when $s\to 1$ to the classical problem. We also study the continuous dependence with respect to the level-set function $v$, yielding an existence result for a two-phase obstacle-type problem with an implicit level-set function.
\end{abstract}

\section{Introduction}
The aim of this work is to study the continuous dependence and some of its consequences of the following two-phase obstacle-type problem in an arbitrary open bounded set $\Omega\subset\R^d$, formally in the form
\begin{equation}\label{eq:two_phase_equation}
    \begin{cases}
        -\Delta^s_p u+\lambda_+\chi_{\{u>v\}}-\lambda_-\chi_{\{u<v\}}=f &\mbox { in } \Omega\\
        u=0 &\mbox{ on } \R^d\setminus\Omega,
    \end{cases}
\end{equation}
where $-\Delta^s_p$ is the fractional $p$-Laplacian, with $p\in(1,\infty)$ and $s\in(0,1)$. This operator arises naturally in the fractional setting of the Riesz fractional gradient $D^s$ as a generalization of the classical $p$-Laplacian $-\Delta_p$, formally corresponding to the limit case $s=1$, with $D^1=D$ being the classical gradient. Here $f=f(x)$ and $\lambda=\lambda_{\pm}(x)\geq0$ a.e. in $\Omega$ are functions in $L^q(\Omega)$ for suitable $q$, and $\chi_A$ is the characteristic function of the set $A\subset \Omega$. The level set function $v=v(x)$ may be directly given in $L^1(\Omega)$ or implicitly given as a function $v=\Phi[u]$ depending on the solution $u$.

For each given $v$, the solution $u=u^s$ to \eqref{eq:two_phase_equation} can be obtained as the minimizer in the Lions-Calderón space $\Lambda^{s,p}_0(\Omega)$, a fractional Sobolev-type space for $0<s<1$, to the functional
\begin{equation}\label{eq:functional_minimized}
    J^s_p(w)=\frac{1}{p}\int_{\R^d}{|D^s w|^p}\,dx+\int_\Omega{[\lambda_+(w-v)^+ +\lambda_-(w-v)^-]}\,dx -\int_\Omega{fw}\,dx,
\end{equation}
with $w^\pm=\max\{\pm w, 0\}$. We shall show that $u=u^s$ converges in the limit case $s=1$, to $u^1$, which is the minimizer of the functional $J^1_p(w)$ in the classical Sobolev space $W^{1,p}_0(\Omega)=\Lambda^{1,p}_0(\Omega)$, for $1<p<\infty$. However, these variational solutions do not always satisfy the equation (1) a.e. in $\Omega$ but rather the following weaker form of that equation

\begin{equation}\label{eq:weak_formulation_with_quasicharacteristic}
    -\Delta^s_p u+\lambda_+\chi_{{}_{>}}-\lambda_-\chi_{{}_{<}} =f \quad \mbox { in } \Omega,
\end{equation}
with the characteristic functions $\chi_{\{u>v\}}$ and $\chi_{\{u<v\}}$ replaced, respectively, by quasi-characteristic functions $\chi_{{}_{>}}$ and $\chi_{{}_{<}}$, such that
\begin{equation}\label{eq:quasicharacteristic}
    \chi_{\{u>v\}}\leq \chi_{{}_{>}}\leq \chi_{\{u\geq v\}}, \quad \chi_{\{u<v\}}\leq \chi_{{}_{<}}\leq \chi_{\{u\leq v\}} \quad \mbox { in } \Omega.
\end{equation}

This type of equations arise in stationary models of interior heat control \cite{duvaut1976inequalities} and in composite membranes, as referred in \cite{petrosyan2012regularity}, which presents an extended exposition, with further references, of the regularity of the free boundaries
\begin{equation}\label{eq:free_boundaries}
    \Gamma_{<v}=\partial\{u<v\}\cap\Omega\quad \mbox{ and }\quad \Gamma_{>v}=\partial\{u>v\}\cap\Omega,
\end{equation}
in the classical case of the Laplacian, with $v=0$ and $f=0$. The free boundaries $\Gamma_{<0}$ and $\Gamma_{>0}$ may, in general, present branching points when they start intersecting the boundary of the degenerate phase $\partial\{u=0\}$, whenever this intermediate phase, the inter-phases, has non empty interior. In this case the parts of the free boundaries $\Gamma_{<0}\cap \partial \{u=0\}$ and $\Gamma_{>0}\cap \partial \{u=0\}$ behave locally as the free boundaries of the upper and lower obstacle problems, respectively. For nonlinear operators of $p$-Laplacian type, in the classical two-phase problem for $v=f=0$ and constant $\lambda_{\pm}>0$, it has been shown that the free boundaries have locally finite $(d-1)$-Hausdorff measure for $p$ close to 2 (for $p>2$ in \cite{edquist2009pLaplacian} and for $p<2$ in \cite{zheng2018pLaplacian}). In a nondegenerate case for large $f\neq 0$ and more general quasilinear operators when the inter-phases set has zero Lebesgue measure it has been shown that the free boundary is, up to a null subset, the union of an at most countable family of $C^1$-surfaces in \cite[Theorem 4.2]{rodrigues2016free}.

It was suggested in \cite{bellido2020Fractional} that the fractional $p$-Laplacian can be used to study nonlocal hyperelastic materials, so it might be of interest to consider the problem \eqref{eq:two_phase_equation} to design certain nonlinear and nonlocal composite membranes. On the other hand, in the case $p=2$ it can be used to model equilibrium situations of foraging, since it have been well documented \cite{sims2008scaling} that some predators exhibit spatial Levy-like behaviour when searching for food, which correspond to the fractional Laplacian.

Besides proving the existence of solutions, we also study the continuous dependence of these solutions with respect to the given function $v$ and with respect to $s$ as $s\nearrow 1$. These results have interesting consequences. The continuous dependence of these solutions with respect to the given function $v$ allows us to prove an existence result for implicit two-phase problems formally in the form
\begin{equation}\label{eq:quasi-variational_two_phase_equation}
    \begin{cases}
        -\Delta^s_p u+\lambda_+(x)\chi_{\{u>\Phi(u)\}}-\lambda_-(x)\chi_{\{u<\Phi(u)\}}=f &\mbox { in } \Omega\\
        u=0 &\mbox{ on } \R^d\setminus\Omega.
    \end{cases}
\end{equation}
where $\Phi$ is an operator which depends on the solution of the problem and the characteristic functions are relaxed to quasi-characteristic functions as in \eqref{eq:quasicharacteristic}.

A motivation for studying the convergence of the fractional problem to the local problem, when $s\to 1$, arises from the fact that, in general, the properties of the free boundaries $\Gamma_{<v}$ and $\Gamma_{>v}$ of the two-phase problems for the classical $p$-Laplacian and for fractional differential operators are far from being well-understood. In fact, as it was observed in \cite{oton2023integro}, for the linear fractional case, the structure of the free boundaries is only known for problems that can be transformed, using the Caffarelli-Silvestre (or Stinga-Torrea) extension method, see \cite{caffarelli2007extension, stinga2010extension}, into a degenerate local problem. We remark that it is not known if there is also an extension method for the fractional differential operators that arise in the context of the Riesz fractional gradient, aside from the fractional Laplacian. Nevertheless, we prove here that the solutions to the two phase problem in the fractional framework approximate the classical solutions, which suggests that they may share, at least for $s$ near 1, some of the properties of the classical problem.

After introducing the functional framework in Section \ref{sec:preliminaries}, we present in Section \ref{sec:characterization_solutions} the characterization of the solution to the two-phase obstacle-type problems for $0<s\leq 1$ by using an approximation by $\varepsilon$-regularized semilinear problems. In particular, we give a precise new estimate for the $L^p$-norm of the gradients $D^s$.

In Section \ref{sec:convergence_fractional_to_classical}, we prove the convergence of the solution $u^s$ of the fractional problem to the solution of the classical problem $u^1$ when $s\to 1$, as well as a weak stability of the quasi-characteristic functions $\chi^s_{{}_{>}}$ and $\chi^s_{{}_{<}}$ towards $\chi^1_{{}_{>}}$ and $\chi^1_{{}_{<}}$, respectively. This convergence is in fact strong if the limit problem is non-degenerate, i.e., if $\chi^1_{{}_{>}}=\chi_{\{u_1>v\}}$ and $\chi^1_{{}_{<}}=\chi_{\{u_1<v\}}$, which holds whenever the Lebesgue measure of the inter-phases $\{u_1=v\}$ is zero.

Finally, in Section \ref{sec:quasi-variational_problem}, exploring the continuous dependence of the solution and of the quasi-characteristic functions of the Dirichlet problem \eqref{eq:weak_formulation_with_quasicharacteristic}, we prove the existence of a weak solution to the implicit two-phase problem \eqref{eq:quasi-variational_two_phase_equation}, giving examples of possible continuous operators $\Phi$ with interesting properties, including a special case with a vanishing in measure inter-phase $\{u=\Phi(u)\}$.

We observe that, although our results are developed in the framework of $-\Delta^s_p$, they can be easily generalized to more general quasi-linear operators of the $p$-Laplacian type $-D^s\cdot (\boldsymbol{a}(x,D^s u)$ as in \cite[Section 3.3.1]{campos2023unilateral}.

\section{Preliminaries}\label{sec:preliminaries}
We consider the fractional Riesz gradient for smooth functions with compact support $\varphi$ as it was defined in \cite{shieh2015On},
\begin{equation*}
    D^s \varphi=D(I_{1-s}\varphi)
\end{equation*}
where $I_{1-s}$ is the Riesz potential of order $1-s$, $0<1-s<1$,
\begin{equation*}\label{eq:riez_potential}
    I_{1-s}\varphi(x)=(I_{1-s}*\varphi)(x)=\frac{2^s\Gamma\left(\frac{d+s+1}{2}\right)}{\pi^{d/2}(d-1+r)\Gamma\left(\frac{1-s}{2}\right)}\int_{\R^d}{\frac{\varphi(y)}{|x-y|^{d+s-1}}}\,dy.
\end{equation*}
Although this definition requires a certain regularity of the function $\varphi$, it can be extended to more general classes of functions. In fact, due to the fractional integration by parts formula
\begin{equation}\label{eq:frac_integration_by_parts}
    \int_{\R^d}{\varphi(x) D^s\cdot \Phi(x)}\,dx=-\int_{\R^d}{\Phi(x)\cdot D^s \varphi(x)}\,dx,
\end{equation}
which is first valid for all functions $\varphi\in C^\infty_c(\R^d)$ and $\Phi\in C^\infty_c(\R^d;\R^d)$, we can then extend it to functions $\varphi\in L^p(\R^d)$ for which there exists a function $G\in L^1_\text{loc}(\R^d;\R^d)$ such that
\begin{equation*}
    \int_{\R^d}{\varphi(x) D^s\cdot \Phi(x)}\,dx=-\int_{\R^d}{\Phi(x)\cdot G(x)}\,dx, \quad\forall \Phi\in C^\infty_c(\R^d;\R^d).
\end{equation*}
This distributional characterization of the (weak) fractional gradient allows to identify $G=D^s\varphi$ in the Banach space, the Lions-Calderón space,
\begin{equation*}
    \Lambda^{s,p}_0(\Omega)= \overline{C^\infty_c(\Omega)}^{\|\cdot\|_{\Lambda^{s,p}(\R^d)}},
\end{equation*}
for $0<s<1$ and $1<p<\infty$, with the norm
\begin{equation*}
    f\mapsto \|f\|_{\Lambda^{s,p}(\R^d)}=\left(\|f\|_{L^p(\R^d)}^p+\|D^s f\|_{L^p(\R^d;\R^d)}^p\right)^{1/p}.
\end{equation*}
When $\Omega=\R^d$, this space coincides with the Bessel potential or generalized Sobolev spaces, as they are called in the Harmonic Analysis literature \cite{shieh2015On}, and it also coincides with the space of functions $f\in L^p(\R^d)$, for which there exists a (weak) fractional gradient $D^s f\in L^p(\R^d;\R^d)$, as it was proven in \cite{brue2022AsymptoticsII,kreisbeck2022Quasiconvexity}. When $\Omega\subset\R^d$ is an arbitrary bounded set, as it will be assumed in this work, the functions $f\in\Lambda^{s,p}_0(\Omega)$ will be considered extended to the whole $\R^d$ by zero, their fractional gradient $D^sf$ is different from zero on $\R^d\setminus\Omega$ and we may consider $f\in\Lambda^{s,p}_0(\Omega)$ as a closed subspace of $f\in\Lambda^{s,p}_0(\R^d)$.

Moreover, the identity \eqref{eq:frac_integration_by_parts} can be used, as in the classical case $s=1$, to give a weak formulation to partial differential equations, since it can be extended when the functions on the right hand side, $\varphi$ and $D^s\cdot \Phi$ are in $L^p(\Omega)$ and in $L^{p'}(\R^d)$, respectively, and on the left-hand-side, $\Phi$ and $D^s\varphi$ are in $L^{p'}(\Omega, \R^d)$ and in $L^p(\R^d;\R^d)$, respectively. Consequently, this characterization provides a convenient framework to study fractional partial differential equations through variational methods. In particular, this is useful throughout this article to deal with the fractional $p$-Laplacian,
\begin{equation}
    \Delta^s_p f=D^s\cdot(|D^s f|^{p-2}D^s f),
\end{equation}
since, for $f\in\Lambda^{-s,p'}(\Omega)$,
\begin{equation*}
    \langle -\Delta^s_p f,g\rangle_{\Lambda^{-s,p'}(\Omega)\times\Lambda^{s,p}_0(\Omega)}=\int_{\R^d}{|D^s f|^{p-2}D^s f\cdot D^s g}\,dx, \quad \forall g\in\Lambda^{s,p}_0(\Omega),
\end{equation*}
and $\Delta^s_p f$ is a distribution belonging to the dual space $\Lambda^{-s,p'}(\Omega)=\left(\Lambda^{s,p}_0(\Omega)\right)'$.

This space has some properties that are similar to those observed in the classical Sobolev spaces $W^{1,p}_0(\Omega)\subset\Lambda^{s,p}_0(\Omega)$, for $0<s<1$ and $1<p<\infty$. In fact, as $D^sf\to Df$ when $s\to 1$ for $f\in W^{1,p}_0(\Omega)$, we shall denote the gradient $Df=D^1f$ and $W^{1,p}_0(\Omega)=\Lambda^{1,p}_0(\Omega)$, by considering also $Df=0$ in $\R^d\setminus\Omega$. 

\begin{proposition}[Poincaré's inequality]\label{prop:fractional_poincare_inequality}
    Let $\Omega$ be a bounded open subset of $\R^d$, $p\in(1,\infty)$ and $s\in(0,1]$. There exists a positive constant $C_P>0$ depending only on $p$, $d$, $\Omega$ and $R$, such that for all open sets $\Omega_1\supset B_{2R}(0)$ ($\Omega_1$ can be unbounded) and for all $f\in \Lambda^{s,p}_0(\Omega)$, we have 
    \begin{equation}\label{eq:poincare_inequality}
        \|f\|_{L^p(\Omega)}\leq \frac{C_P}{s}\|D^s f\|_{L^p(\Omega_1;\R^d)}.
    \end{equation}
\end{proposition}
\begin{proof}
    See \cite{bellido2020gamma} and \cite[Theorem 2.4]{campos2023unilateral}.
\end{proof}

As in the classical case, for bounded sets $\Omega$, in $\Lambda^{s,p}_0(\Omega)$ we can work, as we shall do here, with the equivalent norm $f\mapsto \|f\|_{\Lambda^{s,p}_0(\Omega)}=\left(\|D^s f\|_{L^p(\R^d;\R^d)}^p\right)^{1/p}$,  for $0<s\leq 1$ and $1<p<\infty$.

\begin{proposition}[Sobolev embeddings]\label{prop:sobolev_embeddings_lions_calderon}
    Let $0< s\leq 1$, $1<p<\infty$ and $\Omega\subset \R^d$ be an arbitrary open set. Then,
    \begin{itemize}
        \item[1)] if $sp<d$, we have $\Lambda^{s,p}_0(\Omega)\subset L^q(\Omega)$ for $1\leq q\leq p^*_s=\frac{dp}{d-sp}$. Moreover, when $q=p^*_s$ we have
        \begin{equation}\label{eq:sobolev_inequality}
            \|f\|_{L^{p^*_s}(\Omega)}\leq C\|D^s f\|_{L^p(\R^d;\R^d)};
        \end{equation}
        \item[2)] if $sp=d$, then $\Lambda^{s,p}_0(\Omega)\subset L^q(\Omega)$ for every $q\in[1,+\infty)$; and
        \item[3)] if $sp>d$, then we have $\Lambda^{s,p}_0(\Omega)\subset C^{0,s-\frac{d}{p}}(\overline{\Omega})$ with the estimate
        \begin{equation*}
            |f(x)-f(y)|\leq C|x-y|^{s-\frac{d}{p}}\|D^s f\|_{L^p(\R^d)}.
        \end{equation*}
    \end{itemize}
\end{proposition}
\begin{proof}
    See \cite[Theorems 1.8 and 1.11]{shieh2015On}.
\end{proof}

As in the classical $s=1$ case, the Sobolev's embeddings are also compact for all $q$, such that $1\leq q<p^*_s$, that is the Relich-Kondrachov theorem also holds for $0<s<1$.

For future reference, we denote the conjugate exponent $p^\#_s=(p^*_s)'$, where $p^\#_s=dp/[d(p-1)+sp]$ if $sp<d$, $p^\#_s=1$ if $sp>d$ and $p^\#_s>1$ if $sp=d$, for $0<s\leq1$.

\begin{proposition}\label{prop:fractional_rellich_kondrachov}
    Let $\Omega\subset \R^d$ be a bounded open set, and let $0\leq t<s\leq 1$ and $p\in(1,\infty)$. Then,
    \begin{equation*}
        \Lambda^{s,p}_0(\Omega)\Subset \Lambda^{t,p}_0(\Omega).
    \end{equation*}
\end{proposition}

Moreover, for our results about the approximation of solutions and the continuous dependence of the solutions, as $s\to \sigma=1$, we shall use the following two properties.

\begin{proposition}\label{prop:compactness_ala_bellido}
    Let $\Omega$ be a bounded open subset of $\R^d$ and $p\in(1,+\infty)$. Let $s_n\in(0,1]$ be a convergent sequence of numbers such that $s_n\to\sigma\in(0,1]$, and $u_n\in \Lambda^{s_n,p}_0(\Omega)$ be such that $\sup_n{\|D^{s_n} u_n\|_{L^p(\R^d;\R^d)}}<\infty$. Then, there exists a function $u\in\Lambda^{\sigma,p}_0(\Omega)$ and a subsequence also denoted by $u_n$ satisfying
    \begin{equation*}
        u_n\to u \mbox{ in } \Lambda^{t,p}_0(\Omega), \mbox{ for all } 0\leq t<\sigma, \quad \mbox{ and }\quad D^{s_n} u_n\rightharpoonup D^\sigma u \mbox{ in } L^p(\R^d;\R^d).
    \end{equation*}
\end{proposition}
\begin{proof}
    See \cite{bellido2020gamma} for the case $\sigma=1$ and $t=0$. One can use Proposition \ref{prop:fractional_rellich_kondrachov} as in \cite{campos2023unilateral} to generalize to $0<t<\sigma\leq 1$.
\end{proof}

\begin{proposition}\label{prop:lower_estimate_p_laplacian}
    Let $0<s,t\leq 1$, $1<p<\infty$ and assume that $u\in\Lambda^{s,p}_0(\Omega)$ and $v\in \Lambda^{t,p}_0(\Omega)$. Then, with $\alpha_p=2^{2-p}$ for $p\geq 2$ and $\alpha_p=p-1$ for $1<p\leq 2$,
    \begin{multline*}
        \int_{\R^d}{(|D^s u|^{p-1}D^su-|D^t v|^{p-1}D^tv)\cdot D^s u-D^t v)}\,dx\\
        \geq
        \begin{cases}
            \alpha_p \|D^s u-D^t v\|^p_{L^p(\R^d;\R^d)} &\mbox{ if } p\geq 2\\
            \alpha_p \frac{\|D^s u-D^t v)\|^2_{L^p(\R^d;\R^d)}}{(\|D^s u\|_{L^p(\R^d;\R^d)}+\|D^t v\|_{L^p(\R^d;\R^d)})^{2-p}}, &\mbox{ if } 1<p\leq 2.
        \end{cases}
    \end{multline*}
\end{proposition}
\begin{proof}
    This is a simple application of \cite[Lemma 1.11]{edmunds2023Fractional} in the vectorial case.
\end{proof}

\section{Characterization of the solution to the two-phase problem}\label{sec:characterization_solutions}
Throughout this section we shall assume 
\begin{equation}\label{eq:v_in_q_prime}
    v\in L^{q'}(\Omega)
\end{equation}
with $q'=q/(q-1)$, $1\leq q'<p^*_s$, for some $q$, $p^\#_s<q\leq \infty$, and
\begin{equation}\label{eq:nondegeneracy_lambda}
    \lambda_+, \lambda_-\in L^q(\Omega)\mbox{ with }\lambda_\pm\geq 0 \mbox{ a.e. }x\in\Omega, 
\end{equation}
where $p^\#_s=(p^*_s)'$ is the Sobolev conjugate exponent, i.e. $p^\#_s=dp/[d(p-1)+sp]$ if $sp<d$, $p^\#_s=1$ if $sp>d$ and $p^\#_s>1$ if $sp=d$, for $0<s\leq1$.

We define the convex functional $\Psi_v(w)$ in $\Lambda^{s,p}_0(\Omega)$, for $0<s<1$ and in $W^{1,p}_0(\Omega)=\Lambda^{1,p}_0(\Omega)$ for $s=1$,

\begin{equation*}
    \Psi_v(w)=\int_\Omega{(\lambda_+(x)(w-v)^+ +\lambda_-(x)(w-v)^-)}\,dx
\end{equation*}

\begin{proposition}\label{prop:variational_inequality}
    Assume $f\in \Lambda^{-s,p'}(\Omega)=(\Lambda^{s,p}_0(\Omega))'$. Then for each s, $0<s\leq 1$, there exists a unique solution $u=u_s\in \Lambda^{s,p}_0(\Omega)$ to
    \begin{equation}\label{eq:existence_solution}
        \Psi_v(w)-\Psi_v(u)\geq \langle f+\Delta^s_p u, w-u\rangle_{\Lambda^{-s,p'}(\Omega)\times \Lambda^{s,p}_0(\Omega)}\quad \forall w\in \Lambda^{s,p}_0(\Omega).
    \end{equation}
    We can express the variational inequality \eqref{eq:existence_solution} in terms of subdifferentials, by stating that there exists $\zeta_v\in \Lambda^{-s,p'}(\Omega)$ such that
    \begin{equation}\label{eq:existence_solution_subdifferential}
        \zeta_v=f+\Delta^s_p u\in \partial \Psi_v(u)
    \end{equation}
\end{proposition}
\begin{proof}
    This is a simple application of the direct method of calculus of variations to the minimisation problem for the strictly convex and l.s.c. functional defined in \eqref{eq:functional_minimized}, which is coercive by Poincaré inequality:
    \begin{equation*}
        u\in\Lambda^{s,p}_0(\Omega):\quad J^s_p(u)\leq J^s_p(w),\quad\forall w\in\Lambda^{s,p}(\Omega).
    \end{equation*}
    Decomposing the functional $J^s_p$ in the differentiable component plus  $\Psi_v$, we obtain the Euler variational inequality \eqref{eq:existence_solution} for the minimal solution $u$.
\end{proof}

To characterize the function $\zeta_v$ in the case $f\in L^{p^\#_s}(\Omega)\subset \Lambda^{-s,p'}(\Omega)$, we study the following regularised problem for $\varepsilon>0$:
\begin{equation}\label{eq:regularised_problem}
    \begin{cases}
        -\Delta^s_p u_\varepsilon+\lambda_+H_\varepsilon(u_\varepsilon-v)-\lambda_-H_\varepsilon(v-u_\varepsilon)=f &\mbox { in } \Omega\\
        u_\varepsilon=0 &\mbox{ on } \R^d\setminus\Omega.
    \end{cases}
\end{equation}
Here, the characteristic functions $\chi_{\{u>v\}}$ and $\chi_{\{u<v\}}$ that appear in \eqref{eq:two_phase_equation} are replaced by $G'_\varepsilon(u-v)=H_\varepsilon(u-v)$ and $G'_\varepsilon(v-u)=H_\varepsilon(v-u)$, respectively, where
\begin{equation*}
    H_\varepsilon(t)=\begin{cases}
        0, & t\leq 0\\
        \frac{t}{\varepsilon}, & 0\leq t\leq \varepsilon\\
        1, & t\geq \varepsilon
    \end{cases}
    \quad \mbox{ and }\quad
    G_\varepsilon(t)=\begin{cases}
        0, & t\leq 0\\
        \frac{t^2}{2\varepsilon}, & 0\leq t\leq \varepsilon\\
        t-\frac{\varepsilon}{2}, & t\geq \varepsilon.
    \end{cases}
\end{equation*}

\begin{theorem}\label{thm:exist_and_conv_regula_solut}
    Under the assumption $f\in L^{p^\#_s}(\Omega)$, for each s, $0<s\leq 1$, as $\varepsilon \to 0$ the  unique solution $u_\varepsilon\in\Lambda^{s,p}_0(\Omega)$ to \eqref{eq:regularised_problem} converges in $\Lambda^{s,p}_0(\Omega)$ towards the unique solution $u$ of \eqref{eq:existence_solution} and satisfies the estimate
     \begin{equation}\label{eq:convergence_gradients}
        \|D^s(u_\varepsilon-u)\|_{L^p(\R^d;\R^d)}\leq C\varepsilon^{1/\max\{p,2\}}.
    \end{equation}
\end{theorem}
\begin{proof}
    Step 1: Approximating problem and a-priori estimates.

    By introducing the differentiable convex functional   
    \begin{equation*}
        J_{p,\varepsilon}^s(w)=\frac{1}{p}\int_{\R^d}{|D^s w|^p}\,dx+ \int_\Omega{[\lambda_+G_\varepsilon(w-v) +\lambda_-G_\varepsilon(v-w)]}\,dx-\int_\Omega{fw}\,dx,
    \end{equation*}
    we easily see that the unique minimising solution $u_\varepsilon$ of $J_{p,\varepsilon}^s$ in $\Lambda^{s,p}_0(\Omega)$ satisfies a.e. the equation \eqref{eq:regularised_problem}. Indeed this equation in variational form corresponds to the Euler equation of the regularised minimum problem, firstly in the duality $\langle \cdot, \cdot \rangle_{\Lambda^{-s,p'}(\Omega)\times \Lambda^{s,p}_0(\Omega)}$. As a consequence, by the assumption on $f$ and the Sobolev inbeddings, we have $\int_\Omega{fw}\,dx=\langle f, w\rangle_{\Lambda^{-s,p'}(\Omega)\times \Lambda^{s,p}_0(\Omega)}$ for all $w\in \Lambda^{s,p}_0(\Omega)$, and the equation \eqref{eq:regularised_problem} also holds a.e. in $\Omega$.
 
    Now we observe that $u_\varepsilon$ is uniformly bounded in $\varepsilon >0$. Indeed, denoting 
    \begin{equation}\label{eq:def_zeta_epsilon}
        \zeta_\varepsilon(x)=\lambda_+(x)H_\varepsilon(u_\varepsilon-v)(x)-\lambda_-(x) H_\varepsilon(v-u_\varepsilon)(x),
    \end{equation}
    then we have
    \begin{equation*}
        -\lambda_-(x)\leq \zeta_\varepsilon(x)\leq \lambda_+(x) \mbox{ for a.e. } x\in\Omega,
    \end{equation*}
    which together with $f\in L^{p^\#_s}(\Omega)$, yields $-\Delta^s_p u_\varepsilon=\zeta_\varepsilon-f\in  L^{p^\#_s}(\Omega)$ uniformly for $\varepsilon>0$. Multiplying both sides of the previous identity with $u_\varepsilon$, extended by zero in $\R^d$ and applying the formula \eqref{eq:frac_integration_by_parts}, with the help of Proposition \ref{prop:sobolev_embeddings_lions_calderon} and the fractional Sobolev's inequality \eqref{eq:sobolev_inequality}, we obtain
    \begin{equation}\label{eq:estimate_for_frac_grad_regularised_solution}
        \|D^s u_\varepsilon\|^{p-1}_{L^p(\R^d;\R^d)}\leq C(\Omega,p,s)(\|\lambda_+\|_{L^q(\Omega)}+\|\lambda_-\|_{L^q(\Omega)}+\|f\|_{L^{p^\#_s}(\Omega)}).
    \end{equation}
    From Proposition \ref{prop:fractional_rellich_kondrachov}, as $\varepsilon\to 0$, we can extract a subsequence from $u_\varepsilon$ such that
    \begin{equation*}
        u_\varepsilon\to u\mbox{ in } L^p(\Omega)\quad\mbox{ and }\quad D^s u_\varepsilon\rightharpoonup D^s u \mbox{ in } L^p(\R^d;\R^d)
    \end{equation*}
    for some function $u\in \Lambda^{s,p}_0(\Omega)$.

    \vspace{2mm}

    Step 2: Characterization of the limit.

    We are going to show that there exists a function $\zeta\in L^q(\Omega)$ such that $\zeta=f+\Delta^s_p u$. Indeed, we notice that the pointwise estimates $0\leq H_\varepsilon(t)\leq 1$, for all $t\in\R$, together with weak-$*$ compactness of closed balls in $L^\infty$ (see, for instance, Theorem 3.16 and page 101 of \cite{brezis2011functional}) implies the existence of functions $\chi_{{}_{>}},\chi_{{}_{<}}\in L^\infty(\Omega)$, with $0\leq \chi_{{}_{>}}(x), \chi_{{}_{<}}(x)\leq 1$ for a.e. $x\in\Omega$, such that
    \begin{equation}\label{eq:weak_star_convergence}
        H_\varepsilon(u_\varepsilon-v)\overset{\ast}{\rightharpoonup} \chi_{{}_{>}} \mbox{ and } H_\varepsilon(v-u_\varepsilon)\overset{\ast}{\rightharpoonup} \chi_{{}_{<}} \mbox{ in } L^\infty(\Omega).
    \end{equation}
    Observe that $\chi_{{}_{>}}=0$ almost everywhere in $\{u<v\}$ because
    \begin{equation*}
        0=\lim_{\varepsilon\to 0}{\int_{\Omega}{H_\varepsilon(u_\varepsilon-v)(u_\varepsilon-v)^-}\,dx}=\int_{\Omega}{\chi_{{}_{>}} (u-v)^-}\,dx=\int_{\{u<v\}}{\chi_{{}_{>}} (u-v)}\,dx.
    \end{equation*}
    In addition, $\chi_{{}_{>}}=1$ almost everywhere in $\{u>v\}$ since
    \begin{align*}
        0
        &=\lim_{\varepsilon\to 0}{\int_{\{v<u_\varepsilon<v+\varepsilon\}}{(1-\frac{u_\varepsilon-v}{\varepsilon})(u_\varepsilon-v)}\,dx}\\
        &=\lim_{\varepsilon\to 0}{\int_{\Omega}{(1-H_\varepsilon(u_\varepsilon-v))(u_\varepsilon-v)^+}\,dx}
        =\int_{\Omega}{(1-\chi_{{}_{>}}) (u-v)^+}\,dx\\
        &=\int_{\{u>v\}}{(1-\chi_{{}_{>}}) (u-v)}\,dx
    \end{align*}
    A similar argument can be used to prove that $\chi_{{}_{<}}=0$ almost everywhere in $\{u>v\}$ and $\chi_{{}_{<}}=1$ almost everywhere in $\{u<v\}$. Consequently,
    \begin{equation*}
        \chi_{\{u>v\}}\leq \chi_{{}_{>}}\leq \chi_{\{u\geq v\}}, \quad \chi_{\{u<v\}}\leq \chi_{{}_{<}}\leq \chi_{\{u\leq v\}}
    \end{equation*}
    and, since $q>p^\#_s$,
    \begin{equation*}
        \zeta_\varepsilon\rightharpoonup\zeta\mbox{ in } L^q(\Omega) \mbox{ weak and also in } \Lambda^{-s,p'}(\Omega) \mbox{ strong },
    \end{equation*}
    where
    \begin{equation}\label{eq:def_zeta}
        \zeta=\lambda_+\chi_{{}_{>}}-\lambda_-\chi_{{}_{<}}.
    \end{equation}
    Moreover, by using the weak formulation of \eqref{eq:regularised_problem} with $w\in \Lambda^{s,p}_0(\Omega)$,
    \begin{equation*}
        \int_{\R^d}{|D^s u_\varepsilon|^{p-2}D^s u_\varepsilon\cdot D^s w}\,dx+\int_\Omega{\zeta_\varepsilon w}\,dx=\int_\Omega{fw}\,dx.
    \end{equation*}
    and letting $\varepsilon\to 0$, we deduce
    \begin{equation*}
        \int_{\R^d}{|D^s u|^{p-2}D^s u\cdot D^s w}\,dx+\int_\Omega{\zeta w}\,dx=\int_\Omega{fw}\,dx.
    \end{equation*}
    
    In order to show that $\zeta=\zeta_v\in\Psi_v(u)$, we define the functional
    \begin{equation*}
        \Psi^\varepsilon(w)=\int_\Omega{[\lambda_+G_\varepsilon(w-v)+\lambda_-G_\varepsilon(v-w)]}\,dx.
    \end{equation*}
    On the one hand, we have $\lim_{\varepsilon\to 0}{\Psi^\varepsilon(w)}=\Psi_v(w)$,  for all $w\in\Lambda^{s,p}_0(\Omega)$, while on the other we have
    \begin{equation*}
        \liminf_{\varepsilon\to 0}{\Psi^\varepsilon(u_\varepsilon)}\geq\lim_{\delta\to 0}{\Psi^\delta(u)}=\Psi_v(u).
    \end{equation*}
    since $G_\delta(t)\leq G_\varepsilon(t)$ for each $t\in\R$ and $0<\varepsilon<\delta$, and 
    \begin{equation*}
        \liminf_{\varepsilon\to 0}{\Psi^\varepsilon(u_\varepsilon)}\geq \liminf_{\varepsilon\to 0}{\Psi^\delta(u_\varepsilon)}=\Psi^\delta(u).
    \end{equation*}
    Consequently, from the definition of $\zeta_\varepsilon\in\partial\Psi^\varepsilon(u_\varepsilon)=\{(\Psi^\varepsilon)'(u_\varepsilon)\}$, we deduce
    \begin{equation*}
        \Psi_v(w)+\langle \zeta, w-u\rangle=\lim{[\Psi^\varepsilon(w)+\langle\zeta^\varepsilon, w-u_\varepsilon\rangle]}\geq\liminf{\Psi^\varepsilon(u_\varepsilon)}\geq \Psi_v(u)
    \end{equation*}
   for all $w\in\Lambda^{s,p}_0(\Omega)$, implying that $\zeta=\zeta_v\in\partial\Psi_v(u)\subset \Lambda^{-s,p'}(\Omega)$.

    \vspace{2mm}

    Step 3: Strong convergence.

    To conclude, we prove the convergence of $D^s u_\varepsilon$ to $D^s u$ in $L^p(\R^d;\R^d)$. If to the identity \eqref{eq:existence_solution_subdifferential}, we subtract \eqref{eq:regularised_problem}, and use the definitions \eqref{eq:def_zeta_epsilon} and \eqref{eq:def_zeta}, of $\zeta_\varepsilon$ and $\zeta_v$ respectively, then
    \begin{multline*}
        \int_{\R^d}{\left(|D^s u|^{p-2}D^s u-|D^s u_\varepsilon|^{p-2}D^s u_\varepsilon\right)\cdot D^s(u-u_\varepsilon)}\,dx
        =\int_{\Omega}{(\zeta_\varepsilon-\zeta)(u-u_\varepsilon)}\,dx\\
        =\int_\Omega{\lambda_+[H_\varepsilon(u_\varepsilon-v)-\chi_{{}_{>}}](u-u_\varepsilon)}\,dx-\int_\Omega{\lambda_-[H_\varepsilon(v-u_\varepsilon)-\chi_{{}_{<}}](u-u_\varepsilon)}\,dx.
    \end{multline*}
    Since for every $t\geq 0$, $(1-H_\varepsilon(t))t\leq \varepsilon$, and $0\leq \chi_{\{u>v\}}\leq \chi_{{}_{>}}\leq \chi_{\{u\geq v\}}$ a.e. in $\Omega$, we deduce for the positive phase
    \begin{align*}
        &\int_\Omega{\lambda_+[H_\varepsilon(u_\varepsilon-v)-\chi_{{}_{>}}](u-u_\varepsilon)}\,dx\\
        &\leq \int_{\{u>v\}}{\lambda_+[H_\varepsilon(u_\varepsilon-v)-1](u-u_\varepsilon)}\,dx+\int_{\{u=v\}}{\lambda_+[H_\varepsilon(u_\varepsilon-v)-\chi_{{}_{>}}](v-u_\varepsilon)}\,dx\\
        &\leq \int_{\{u>v\}}{\lambda_+[H_\varepsilon(u_\varepsilon-v)-1](v-u_\varepsilon)}\,dx+\int_{\{u=v\}\cap\{u_\varepsilon<v\}}{\lambda_+ \chi_{{}_{>}} (u_\varepsilon-v)}\,dx\\
        &\qquad\qquad+\int_{\{u=v\}\cap\{u_\varepsilon>v\}}{\lambda_+[H_\varepsilon(u_\varepsilon-v)-1](v-u_\varepsilon)}\,dx\\
        &\leq\varepsilon\|\lambda_+\|_{L^1(\Omega)}.
    \end{align*}
    For the negative-phase, we can apply the same reasoning and deduce that 
    \begin{equation*}
        \int_\Omega{\lambda_-[H_\varepsilon(v-u_\varepsilon)-\chi_{{}_{<}}](u-u_\varepsilon)}\,dx\leq\varepsilon\|\lambda_-\|_{L^1(\Omega)}.
    \end{equation*}
    Consequently, for $p\geq2$ we get the uniform estimate
    \begin{multline*}
        \alpha_p \|D^s(u_\varepsilon-u)\|^p_{L^p(\R^d;\R^d)}\leq \int_{\R^d}{\left(|D^s u|^{p-2}D^s u-|D^s u_\varepsilon|^{p-2}D^s u_\varepsilon\right)\cdot D^s(u-u_\varepsilon)}\,dx\\
        \leq\varepsilon(\|\lambda_+\|_{L^1(\Omega)}+\|\lambda_-\|_{L^1(\Omega)}).
    \end{multline*}
    On the other hand, for $1<p<2$, we use the estimate \eqref{eq:estimate_for_frac_grad_regularised_solution} to conclude that there exists a positive constant $C>0$, depending on $s,\Omega,p,\lambda_\pm$ and $f$ such that
    \begin{equation*}
        \|D^s(u_\varepsilon-u)\|^2_{L^p(\R^d;\R^d)}\leq \varepsilon C.
    \end{equation*}
    \end{proof}

As a corollary of this theorem, we obtain the characterization of the solution to the variational inequality \eqref{eq:existence_solution} as the solution of the two-phase problem \eqref{eq:two_phase_equation} together with the characterization of the Lagrange multiplier $\zeta_v$ given by \eqref{eq:characterization_lagrange_multiplier}, as stated in the following theorem. Noted that $\chi_{\{u\geq v\}}=1-\chi_{\{u<v\}}$ and $\chi_{\{u\leq v\}}=1-\chi_{\{u>v\}}$ a.e. in $\Omega$.

\begin{theorem}\label{thm:convergence_solutions_regularised_problem}
    Under the hypothesis of Theorem \ref{thm:exist_and_conv_regula_solut}, there exist functions $(u, \chi_{{}_{>}}, \chi_{{}_{<}})\in \Lambda^{s,p}_0(\Omega)\times L^\infty(\Omega)\times L^\infty(\Omega)$, such that
    \begin{equation}\label{eq:condition_chi}
        \chi_{\{u>v\}}\leq \chi_{{}_{>}}\leq \chi_{\{u\geq v\}}, \quad \chi_{\{u<v\}}\leq \chi_{{}_{<}}\leq \chi_{\{u\leq v\}}\quad \mbox{ a.e. in } \Omega,
    \end{equation}
    that solve the Dirichlet problem
    \begin{equation}\label{eq:two_phase_equation}
        \begin{cases}
            -\Delta^s_p u+\lambda_+\chi_{{}_{>}}-\lambda_-\chi_{{}_{<}}=f &\mbox { in } \Omega\\
            u=0 &\mbox{ on } \R^d\setminus\Omega.
        \end{cases}
    \end{equation}
    Moreover $u$ is the unique solution to \eqref{eq:existence_solution} and 
    \begin{equation}\label{eq:characterization_lagrange_multiplier}
        \zeta_v=\lambda_+\chi_{{}_{>}}-\lambda_-\chi_{{}_{<}}\in L^q(\Omega)
    \end{equation}
    is the unique $\zeta_v\in \partial \Psi_v(u)$.
\end{theorem}
\begin{proof}
    Letting $\varepsilon\to 0$ in the regularized problem \eqref{eq:regularised_problem}, the characterization of the limits in step 2 of the proof of Theorem \ref{thm:exist_and_conv_regula_solut} yields the conclusion.
\end{proof}
\begin{remark}
    Although $u=u_s$, $0<s\leq1$, and $\zeta_v$ are uniquely defined, the functions $\chi_{{}_{>}}, \chi_{{}_{<}}$ do not need to be unique, nor the characteristic functions of the sets $\{u_s>v\}$ and $\{u_s<v\}$, respectively. This is a consequence of the fact that, in general, $\Xi^s_v=\{u_s=v\}$ may have a positive Lebesgue measure $\mathcal{L}^d(\Xi^s_v)> 0$ and $\chi_{{}_{>}}$ and $\chi_{{}_{<}}$ may take any values between $0$ and $1$ in that set $\Xi^s_v$. In case when the inter-phases has zero Lebesgue measure, i.e. if
    \begin{equation}\label{eq:nondegeneracy_condition_v}
        \mathcal{L}^d(\Xi^s_v)=\text{meas}(\{u_s=v\})=0
    \end{equation}
    then we may conclude
    \begin{equation}
        \chi_{\{u_s>v\}}=\chi_{{}_{>}}=\chi_{\{u_s\geq v\}}, \quad \chi_{\{u_s<v\}}=\chi_{{}_{<}}=\chi_{\{u_s\leq v\}}\quad \mbox{ a.e. in } \Omega,
    \end{equation}
    and we have the uniqueness of the functions $\chi_{{}_{>}}$ and  $\chi_{{}_{<}}$.
\end{remark}

\begin{remark}
    The nondegeneracy condition \eqref{eq:nondegeneracy_condition_v} is quite difficult to guarantee and is related to the regularity of the solution $u$ and to the assumptions on the data. In the fractional case $0<s<1$, nothing is known for $p\neq 2$, while when $p=2$, $f\in L^\infty(\Omega)$ and $\Omega$ is $C^{1,1}$, a direct application of \cite[Proposition 1.1 and Theorem 1.2]{oton2014boundary} implies that $u\in C^s(\R^d)$ and $u/d^\alpha\in C^\alpha(\overline{\Omega})$ for some $\alpha\in(0,1)$. These results can be extended to operators of the form $-\mathrm{div}_s(A D^s u)$, with $A:\R^d\to \R^{d\times d}$  sufficiently smooth and strongly elliptic, by a simple application of \cite[Corollary 6.11]{abels2023fractional}. However this is not sufficient and it is an open problem to give examples implying \eqref{eq:nondegeneracy_condition_v} in the fractional case $0<s<1$, even in the case of the fractional Laplacian.
\end{remark}
  

\begin{remark}\label{nondegenerate}
    In the classical case in which $s=1$ and $v=0$, the nondegeneracy condition \eqref{eq:nondegeneracy_condition_v} can be guaranteed for the $p$-Laplacian under additional regularity assumptions on $f$ and $\lambda_{\pm}$, for instance if they are in $W^{1,p'}_\text{loc}(\Omega)$, combined with the restriction
    \begin{equation}\label{eq:hypothesis_f_lambda}
        \lambda_+(x)<f(x)\mbox{ or } f(x)<-\lambda_-(x) \mbox{ a.e. } x\in\Omega.
    \end{equation}
    In fact, by the local classical regularity theory for the $p$-Laplacian, which can be found in \cite[Theorem 8.1]{giusti2003direct}, \cite[Theorem 6.5]{ladyzhenskaya1968linear}, or \cite{fuchs1990holder, tolksdorf1984Regularity}, one has $\Delta_p u=0 \mbox{ a.e. in } \{u=0\}$, if $\mathcal{L}^d(\Xi^1_0)>0$ holds we obtain a contradiction with the condition \eqref{eq:hypothesis_f_lambda}, as it was observed in Section 3 of \cite{rodrigues2016free}. In fact, if the function $\lambda_+ +\lambda_-\in C(\Omega)$, locally in $\omega=\{\lambda_+ +\lambda_->0\}$, in the nondegenerate case $\mathcal{L}^d(\Xi^1_0\cap\omega)=0$, the free boundary is, up to a set of a null perimeter, the union of an at most countable family of $C^1$-surfaces \cite[Theorem 4.2]{rodrigues2016free}.

    In the linear case, with $p=2$ and $v$ such that $\Delta v\in L^\infty(\Omega)$, which implies that the solution $u\in W^{2,r}_\text{loc}(\Omega)$, for all $r<\infty $, and $f+\Delta v$ satisfies the nondegeneracy condition
    \begin{equation}\label{eq:sufficient_condition_nondegeneracy_with_translation}
        \lambda_+(x)<f(x)+\Delta v(x)\mbox{ or } f(x)+\Delta v(x)<-\lambda_-(x) \mbox{ a.e. } x\in\Omega,
    \end{equation}
    a similar reasoning can be applied to conclude that $\mathcal{L}^d(\Xi^1_v)=0$. However, when $s<1$, such arguments are no longer possible, since the fractional operators are nonlocal.
\end{remark} 

\begin{remark}\label{rem:paired_system}
    The approximation and characterisation of the two-phase two membranes problem follows the steps of the previous one membrane case. In fact, considering the minimisation of the vector functional
    \begin{equation}
        \begin{split}
            \mathcal{J}^s(w,z)=\frac{1}{p}\int_{\R^d}&{(|D^s w|^p+|D^s z|^p)}\,dx\\
            &+\int_\Omega{[\lambda_+(w-z)^+ +\lambda_-(w-z)^- -fw-gz]}\,dx,
        \end{split}
    \end{equation}
    in $\Lambda^{s,p}_0(\Omega)\times\Lambda^{s,p}_0(\Omega)$, with $(f,g)\in L^{p^\#_s}(\Omega)\times L^{p^\#_s}(\Omega)$ and $\lambda_\pm$ satisfying \eqref{eq:nondegeneracy_lambda}, similarly we obtain the unique solution $(u,v)$ to the system
    \begin{equation}
        \begin{cases}
            -\Delta^s_p u+\lambda_+\chi_{{}_{>}}-\lambda_-\chi_{{}_{<}}=f &\mbox { in } \Omega\\
            -\Delta^s_p v+\lambda_-\chi_{{}_{<}}-\lambda_+\chi_{{}_{>}}=g &\mbox { in } \Omega\\
            u=0 \mbox{ and } v=0 &\mbox{ on } \R^d\setminus\Omega,
        \end{cases}
    \end{equation}
    with $\chi_{{}_{>}}=\chi_{{}_{>}}[u,v]$ and $\chi_{{}_{<}}=\chi_{{}_{<}}[u,v]$ satisfying \eqref{eq:condition_chi}. The approximation by $(u_\varepsilon,v_\varepsilon)$ can also be done with the corresponding regularised system \eqref{eq:regularised_problem} with the equation for $v_\varepsilon$ with $g$ and the $\lambda_\pm H_\varepsilon(\pm\,\cdot)$ exchanged. The estimate \eqref{eq:convergence_gradients} also holds for $v_\varepsilon$ exactly with the same constant $C$. Although $\chi_{{}_{>}}$ and $\chi_{{}_{<}}$ are, in general also non unique, the Lagrange multiplier $\zeta=\lambda_+\chi_{{}_{>}}-\lambda_-\chi_{{}_{<}}$ is also unique.
\end{remark}

\section{Stability of solutions as $s\nearrow 1$}\label{sec:convergence_fractional_to_classical}
As it was mentioned at the end of previous Section, in general $\chi^s_{{}_{>v}}$ and $\chi^s_{{}_{<v}}$ are not characteristic functions. However, in this section we prove that they have a weak continuous dependence property, when we take $s\nearrow 1$, towards their limits, which in fact holds in the strong topology if they are the characteristic functions in the limit. This happens in the case $s=1$, provided its inter-phases has measure zero,
\begin{equation}\label{eq:nodegeneracy_limit_case}
    \mathcal{L}^d(\Xi^1_v)=\text{meas}(\{u=v\})=0.
\end{equation}

\begin{theorem}\label{thm:convergence_measure_fro_frac_to_cla}
    Let $s\to 1$,  $0\leq \lambda^s_\pm, f_s \in L^{p'}(\Omega)$ and $v_s\in L^p(\Omega)$ be such that $f_s\to f$ and $\lambda^s_\pm\to\lambda_\pm$, both in $L^{p'}(\Omega)$, and  $v_s\to v$ in $L^p(\Omega)$, for  $p'=p/(p-1)>p^\#_1$.  Let $(u_s,\chi^s_{{}_{>}}[v_s], \chi^s_{{}_{<}}[v_s])$ be a solution to the two-phase problem \eqref{eq:two_phase_equation} for each $s$, $0<\sigma<s<1$. Then, for any $0\leq t<1$, there exists a subsequence $s\to 1$, 
    \begin{equation}\label{eq:convergence_u_s}u_s\to u \mbox{ in } \Lambda^{t,p}_0(\Omega),\quad D^su_s\to Du \mbox{ in } L^p(\R^d;\R^d),
    \end{equation}
    and
    \begin{equation}\label{eq:convergence_quasi_characteristics}
        \chi^s_{{}_{>}}[v_s]\overset{\ast}{\rightharpoonup}\chi^1_{{}_{>}}[v] \quad\mbox{ and }\quad \chi^s_{{}_{<}}[v_s]\overset{\ast}{\rightharpoonup}\chi^1_{{}_{<}}[v] \quad\mbox{ in } L^\infty(\Omega), 
    \end{equation}
    where $(u,\chi^1_{{}_{>}}[v], \chi^1_{{}_{<}}[v])\in W^{1,p}_0(\Omega)\times L^\infty(\Omega)\times L^\infty(\Omega)$ is a solution to \eqref{eq:two_phase_equation} associated to $(f, \lambda_\pm, v)$ with $s=1$, i.e. with the $p$-Laplacian.
    In addition, if we assume also that \eqref{eq:nodegeneracy_limit_case} holds, then
    \begin{equation}\label{eq:convergence_quasi_characteristics}
        \chi^s_{{}_{>}}[v_s]\to\chi_{\{u>v\}} \mbox{ and } \chi^s_{{}_{<}}[v_s]\to\chi_{\{u<v\}} \mbox{ strongly in } L^r(\Omega), r<\infty, 
    \end{equation}
    and the whole sequence $(u_s,\chi^s_{{}_{>}}[v_s], \chi^s_{{}_{<}}[v_s])\to (u,\chi_{\{u>v\}}, \chi_{\{u<v\}})$ as $s\to 1$.
\end{theorem}
\begin{proof}
    Since $f_s\to f$ and $\lambda^s_\pm\to\lambda_\pm$ in a $L^{p'}(\Omega)$, and
    \begin{equation*}
        -\lambda^s_-\leq \zeta_s=\lambda^s_+\chi^s_{{}_{>}}[v_s]- \lambda^s_-\chi^s_{{}_{<}}[v_s]\leq \lambda^s_+ \quad \mbox{ a.e. in } \Omega,
    \end{equation*}
    then, using the Poincaré inequality, there exists a constant $M$, depending on the sequences $f_s$ and $\lambda^s_\pm$, but independent of $s>\sigma$, such that
    \begin{equation}\label{eq:estimate_for_frac_grad_regularised_solution}
        \|D^s u_s\|^{p-1}_{L^p(\R^d;\R^d)}\leq \frac{C(\Omega,p)}{s}M < \frac{C(\Omega,p)}{\sigma}M.
    \end{equation}
    From Proposition \ref{prop:compactness_ala_bellido}, we have that $u_s\to u$ in $\Lambda^{t,p}_0(\Omega)$, for $0\leq t<\sigma$, and $D^s u_s\rightharpoonup Du$ in $L^p(\R^d;\R^d)$, for some $u\in W^{1,p}_0(\Omega)$.
    
    Since $0\leq \chi^s_{{}_{>}}[v_s], \chi^s_{{}_{<}}[v_s]\leq 1$, we get from Banach-Allaoglu's theorem that for some subsequence of $s$, there exist functions $\chi^1_{{}_{>}}[v], \chi^1_{{}_{<}}[v]$ such that
    \begin{equation*}
        \chi^s_{{}_{>}}[v_s]\overset{\ast}{\rightharpoonup} \chi^1_{{}_{>}}[v] \mbox{ and } \chi^s_{{}_{<}}[v_s]\overset{\ast}{\rightharpoonup} \chi^1_{{}_{<}}[v] \mbox{ in } L^\infty(\Omega).
    \end{equation*}
    Moreover, since $\chi^s_{{}_{>}}[v_s]=0$ whenever $u_s<v_s$, then
    \begin{equation*}
        0=\int_\Omega{(u_s-v_s)^-\chi^s_{{}_{>}}[v_s]}\,dx\to \int_\Omega{(u-v)^-\chi^1_{{}_{>}}[v]}\,dx
    \end{equation*}
    yielding $0\leq \chi^1_{{}_{>}}[v]\leq \chi_{\{u\geq v\}} \mbox{ a.e. in } \Omega$.

    In addition, if we set $\mathcal{O}\subset\Omega$ to be an arbitrary open set, then 
    \begin{equation*}
        \int_\mathcal{O}{(u_s-v_s)^+ \chi^s_{{}_{>}}[v_s]}\,dx=\int_\mathcal{O}{(u_s-v_s)^+}\,dx\to\int_\mathcal{O}{(u-v)^+}\,dx
    \end{equation*}
    because $\chi^s_{{}_{>}}[v_s]=1$ whenever $u_s>v_s$.
    On the other hand $u_s^+\to u^+$ in $L^p(\Omega)$, which yields
    \begin{equation*}
        \int_\mathcal{O}{(u-v)^+}\,dx=\lim_{s\nearrow 1}{\int_\mathcal{O}{(u_s-v_s)^+\chi^s_{{}_{>}}[v_s]}\,dx}=\int_\mathcal{O}{(u-v)^+\chi_{{}_{>}}^1[v]}\,dx.
    \end{equation*}
    As a consequence, $\chi^1_{{}_{>}}[v]=1$ a.e. in $\{u>v\}$, and so $0\leq \chi_{\{u>v\}}\leq \chi^1_{{}_{>}}[v]$. Similarly, we conclude $0\leq \chi_{\{u<v\}}\leq \chi^1_{{}_{<}}[v]\leq \chi_{\{u\leq v\}}$.
    
    Consider now a sequence $w_s\to w$ in $L^p(\Omega)$. Since $\lambda^s_\pm\to\lambda_\pm$ in $L^{p'}(\Omega)$ and $v_s\to v$ in $L^p(\Omega)$, then
    \begin{multline}\label{eq:convergence_psi_s}
        \lim_{s\to 1}{\Psi^s_{v_s}(w_s)}
        =\lim_{s\to 1}{\int_\Omega{(\lambda^s_+(w_s-v_s)^+ +\lambda^s_-(w_s-v_s)^-)}\,dx}\\
        =\int_\Omega{(\lambda_+(w-v)^+ +\lambda_-(w-v)^-)}\,dx
        =\Psi_v(w).
    \end{multline}
    As a simple consequence of this limit, denoting $\zeta_1=\lambda_+\chi^1_{{}_{>}}[v]- \lambda_-\chi^1_{{}_{<}}[v]$ be the weak limit in $L^{p'}(\Omega)$ of $\zeta_s$ and using the fact that $\zeta_s\in\partial\Psi^s_{v_s}(u_s)$, we get
    \begin{multline*}
        \Psi_v(w)+\int_{\Omega}{\zeta_1(w-u)}\,dx
        =\lim_{s\to 1}{\left[\Psi^s_{v_s}(w)+\int_{\Omega}{\zeta_s(w-u_s)}\,dx\right]}\\
        \geq\lim_{s\to 1}{\Psi^s_{v_s}(u_s)}=\Psi_v(u),\quad\forall w\in W^{1,p}_0(\Omega),
    \end{multline*}
    which implies that $\zeta_1\in\partial\Psi^1_v(u)\subset W^{-1,p'}(\Omega)$. Now, by the monotonicity of $-\Delta^s_p$, we have
    \begin{multline*}
        \int_{\R^d}{|D^s w|^{p-2}D^s w\cdot D^s (w-u_s)}\,dx\geq\\
        \int_{\R^d}{|D^s u^s|^{p-2}D^s u^s\cdot D^s (w-u_s)}\,dx
        =\int_\Omega{(f_s-\zeta_s )(w-u_s)}\,dx.
    \end{multline*}
    Letting $s\nearrow 0$,  since for any $w\in W^{1,p}_0(\Omega)$, $D^sw\to Dw$ in $L^p(\R^d;\R^d)$, we deduce
    \begin{equation*}
        \int_{\Omega}{|Dw|^{p-2}D w\cdot D (w-u)}\,dx\geq \int_\Omega{(f-\zeta_1)(w-u)}\,dx, \quad\forall w\in W^{1,p}_0(\Omega).
    \end{equation*}
    Choosing $w=u\pm \delta z$, for arbitrary $z\in W^{1,p}_0(\Omega)$ and letting $\delta\to 0$, it follows that $(u,\chi^1_{{}_{>}}[v], \chi^1_{{}_{<}}[v])$ is a solution of 
    \eqref{eq:two_phase_equation} associated to $(f, \lambda_\pm, v)$.
    
    By subtracting the equation $-\Delta_p(u)=f-\zeta_1$ to $-\Delta^s_p(u)=f_s-\zeta_s$, and using the fact that $\zeta_s\rightharpoonup\zeta_1$ and $f_s\to f$ in $L^{p'}(\Omega)$ and $u_s\to u$ in $L^p(\Omega)$, we obtain
    \begin{multline*}
        \lim_{s\to 1}{\int_{\R^d}{\left(|D u|^{p-2}D u-|D^s u_s|^{p-2}D^s u_s\right)\cdot (Du-D^su_s)}\,dx}\\
        =\lim_{s\to 1}{\left[\int_{\Omega}{(f-f_s)(u-u_s)}\,dx+\int_{\Omega}{(\zeta_s-\zeta_1)(u-u_s)}\,dx\right]}=0
    \end{multline*}
    From, Proposition \ref{prop:lower_estimate_p_laplacian} follows that $D^s u_s\to Du$ in $L^p(\R^d;\R^d)$ as $s\to 1$.

    Assuming now the nondegeneracy hypothesis $\mathcal{L}^n(\Xi^1_v)=0$, we have
    \begin{equation*}
        \chi^1_{{}_{>}}[v]=\chi_{\{u>v\}}=\chi_{\{u\geq v\}} \quad \mbox { and } \quad \chi^1_{{}_{<}}[v]=\chi_{\{u<v\}}=\chi_{\{u\leq v\}}.
    \end{equation*}
    To get the strong convergence of the phases in $L^r(\Omega)$ , we make use of the weak convergence to get 
    \begin{multline*}
        \int_\Omega{\chi_{\{u>v\}}}\,dx=\lim_{s\to 1}{\int_\Omega{\chi^s_{{}_{>}}[v_s]}\,dx}\geq \limsup_{s\to 1}{\int_\Omega{|\chi^s_{{}_{>}}[v_s]|^q}\,dx}\\
        \geq \liminf_{s\to 1}{\int_\Omega{|\chi^s_{{}_{>}}[v_s]|^q}\,dx}\geq \int_\Omega{|\chi_{\{u>v\}}|^q}\,dx= \int_\Omega{\chi_{\{u>v\}}}\,dx
    \end{multline*}
    Using similar arguments we also have that $\chi^1_{{}_{<}}[v]=\chi_{\{u<v\}}=\chi_{\{u\leq v\}}$ a.e. and the convergence of $\chi^s_{{}_{<}}[v_s]$ to $\chi_{\{u<v\}}$ in $L^r(\Omega)$ with $r<\infty$.
\end{proof}

\begin{corollary}
    Assume the same hypothesis as Theorem \ref{thm:convergence_measure_fro_frac_to_cla}. Then,
    \begin{equation*}
        \chi_{\{u_s>v_s\}}\to\chi_{\{u>v\}} \mbox{ and } \chi_{\{u_s<v_s\}}\to\chi_{\{u<v\}} \mbox{ strongly in } L^q(\Omega), q<\infty. 
    \end{equation*}
\end{corollary}
\begin{proof}
    The same argument used to prove \eqref{eq:convergence_quasi_characteristics} can be applied for this proof. Indeed, we just need to apply it to the functions $\chi_{\{u_s>v_s\}}$ and $\chi_{\{u_s<v_s\}}$, instead of $\chi^s_{{}_{>}}[v_s]$ and $\chi^s_{{}_{<}}[v_s]$.
\end{proof}

\begin{remark}
    As a consequence of \eqref{eq:convergence_psi_s}, in the assumptions of this theorem, we can take the convergences of the  $\lambda^s_\pm$ to $\lambda_\pm$ in $L^q(\Omega)$, for any $q$ such that $p'\leq q\leq \infty$, provided we also assume that $v_s\to v$ in $L^{q'}(\Omega)$, in particular we can have $q'=1$ when $q=\infty$. 
\end{remark}

\begin{remark}
    Similarly for the two-phase two-membrane problem of Remark \ref{rem:paired_system}, when $(f_s,g_s)\to(f,g)$ in $L^{p'}(\Omega)\times L^{p'}(\Omega)$ the conclusions of Theorem \ref{thm:convergence_measure_fro_frac_to_cla} also hold, in particular we have also the convergence $(u_s,v_s,\chi^s_{{}_{>}}[u_s,v_s], \chi^s_{{}_{<}}[u_s,v_s])\to (u,v,\chi^1_{{}_{>}}[u,v], \chi^1_{{}_{<}}[u,v]))$ as $s\to 1$.
\end{remark}

\section{Existence of solutions to the quasi-variational two-phase problem}\label{sec:quasi-variational_problem}
To study the existence of solutions to the quasi-variational inequality \eqref{eq:quasi-variational_two_phase_equation}, first we need to investigate the continuous dependence of the solutions of \eqref{eq:two_phase_equation} with respect to $v$.

\begin{proposition}[Continuous dependence with respect to $v$]\label{prop:continuous_dependence_v}
    Let $0<s\leq 1$, $f\in L^{p^\#_s}(\Omega)$, $\lambda_{\pm}\in L^q(\Omega)$ with $\lambda_\pm\geq 0$ a.e. $x\in\Omega$, and a sequence of functions $v_n\to v$ in $L^{q'}(\Omega)$ as $n\to \infty$, with $q$ and $q'$ as in \eqref{eq:v_in_q_prime} and \eqref{eq:nondegeneracy_lambda}. If we denote by $(u_n, \chi_{{}_{>}}[v_n], \chi_{{}_{<}}[v_n])$ the solution of \eqref{eq:existence_solution_subdifferential} associated to $v_n$ and $(u, \chi_{{}_{>}}, \chi_{{}_{<}})$ the solution of \eqref{eq:existence_solution_subdifferential} associated to $v$, then
    \begin{equation*}
        \chi_{{}_{>}}[v_n]\overset{\ast}{\rightharpoonup} \chi_{{}_{>}}, \chi_{{}_{<}}[v_n]\overset{\ast}{\rightharpoonup} \chi_{{}_{<}} \mbox{ in } L^\infty(\Omega)
    \end{equation*}
    and
    \begin{equation*}
        u_n\to u \mbox{ in } \Lambda^{s,p}_0(\Omega).
    \end{equation*}
\end{proposition}
\begin{proof}
    This follows from a very similar argument to the proof of Theorem \ref{thm:convergence_solutions_regularised_problem}. In fact, we use the uniform bound on $\zeta_n-f$ to be able to extract a subsequence still denoted by $u_n$ such that, $u_n\to u$ in $L^p(\Omega)$, $D^s u_n \rightharpoonup D^s u$ in $L^p(\R^d;\R^d)$ for some function $u\in\Lambda^{s,p}_0(\Omega)$. Then, we make use of the strong convergence in $L^p(\Omega)$ of $u_n$ and $v_n$ to $u$ and $v$, respectively, to characterize the weak limit of $\zeta_n=\lambda_+\chi_{{}_{>}}[v_n]-\lambda_-\chi_{{}_{<}}[v_n]$ to $\zeta=\lambda_+\chi_{{}_{>}}-\lambda_-\chi_{{}_{<}}$ in $L^q(\Omega)$. Together with the weak convergence of $D^s u_n$ implies that $(u,\zeta)$ is the unique solution to the equation \eqref{eq:existence_solution_subdifferential} associated to $v$. Moreover, we also have that
    \begin{equation*}
        \int_{\R^d}{\left(|D^s u|^{p-2}D^s u-|D^s u_n|^{p-2}D^s u_n\right)\cdot D^s(u-u_n)}\,dx
        =\int_{\Omega}{(\zeta_n-\zeta)(u-u_n)}\,dx.
    \end{equation*}
    Since $\zeta_n\rightharpoonup\zeta$ in $L^q(\Omega)$, with $q\geq p'$, and $u_n\to u$ in $L^p(\Omega)\subset L^{q'}(\Omega)$, we get that 
    \begin{equation*}
        \int_{\R^d}{\left(|D^s u|^{p-2}D^s u-|D^s u_n|^{p-2}D^s u_n\right)\cdot D^s(u-u_n)}\,dx\to 0 \mbox{ as } n\to\infty.
    \end{equation*}
    Then, by Proposition \eqref{prop:lower_estimate_p_laplacian}, we get that $D^s u_n\to D^s u$ in $L^p(\R^d;\R^d)$.
\end{proof}

With this result, we can prove, using fixed-point methods, an existence result for the quasi-variational two-phase problem \eqref{eq:quasi-variational_two_phase_equation}.

\begin{theorem}\label{thm:existence_theorem_qvi_two_phase}
    Let $0<s\leq 1$, $f\in L^{p^\#_s}(\Omega)$, $\lambda_{\pm}\in L^q(\Omega)$ with $\lambda_\pm(x)\geq 0$ a.e. $x\in\Omega$ and a continuous application $\Phi:L^{q'}(\Omega)\to L^{q'}(\Omega)$. Then, there exists $(u, \chi_{{}_{>}}[\Phi(u)], \chi_{{}_{< }}[\Phi(u)])\in \Lambda^{s,p}_0(\Omega)\times L^\infty(\Omega)\times L^\infty(\Omega)$ satisfying
    \begin{equation}\label{eq:implicit_problem}
        \begin{cases}
            -\Delta^s_p u+\lambda_+\chi_{{}_{>}}[\Phi(u)]-\lambda_-\chi_{{}_{<}}[\Phi(u)]=f &\mbox { in } \Omega\\
            u=0 &\mbox{ on } \R^d\setminus\Omega.
        \end{cases}
    \end{equation}
    with
    \begin{equation*}
        \chi_{\{u>\Phi(u)\}}\leq \chi_{{}_{>}}[\Phi(u)]\leq \chi_{\{u\geq \Phi(u)\}}, \quad{ and } \quad \chi_{\{u<\Phi(u)\}}\leq \chi_{{}_{<}}[\Phi(u)]\leq \chi_{\{u\leq \Phi(u)\}}
    \end{equation*}
    and $u$ solves the quasi-variational inequality
    \begin{equation}\label{eq:existence_quasi_variational_solution}
        \Psi_{\Phi(u)}(w)-\Psi_{\Phi(u)}(u)\geq \int_\Omega{(f+\Delta^s_p u)(w-u)}\,dx\quad \forall w\in \Lambda^{s,p}_0(\Omega).
    \end{equation}
  
\end{theorem}
\begin{proof}
    Let $L^{q'}(\Omega)\ni v\mapsto S(v)=u_v$ be the solution operator to
    \begin{equation*}
        \begin{cases}
            -\Delta^s_p u_v+\lambda_+\chi_{{}_{>}}[\Phi(v)]-\lambda_-\chi_{{}_{<}}[\Phi(v)]=f &\mbox { in } \Omega\\
            u_v=0 &\mbox{ on } \R^d\setminus\Omega.
        \end{cases}
    \end{equation*}
    Since $|\zeta_{\Phi(v)}|=|\lambda_+\chi_{{}_{>}}[\Phi(v)]-\lambda_-\chi_{{}_{<}}[\Phi(v)]|\leq \lambda_+ +\lambda_-$, from the inequality
    \begin{multline*}
        \|D^s u_v\|^p_{L^p(\R^d;\R^d)}=\langle-\Delta^s_p u_v, u_v\rangle\\
        =\int_\Omega{(f-\zeta_{\Phi(v)})u_v}\,dx\leq C\||f|+\lambda_+ +\lambda_-\|_{L^{p^\#_s}(\Omega)}\|D^s u_v\|_{L^p(\R^d;\R^d)},
    \end{multline*}
    where $C=C(s,p,\Omega)$ is the constant of the Sobolev embedding \eqref{eq:sobolev_inequality},
    it is clear that the solutions $u_v$ satisfy the following estimate independently of $v$,
    \begin{equation*}
        \|u_v\|_{L^{q'}(\Omega)}\leq C'\|D^s u_v\|_{L^p(\R^d;\R^d)}\leq C'(C\||f|+\lambda_+ +\lambda_-\|_{L^{p^\#_s}(\Omega)})^{1/(p-1)}=R
    \end{equation*}
    since $q'<p^*_s$.

    Let $B_R=\{w\in L^{q'}(\Omega): \|w\|_{L^{q'}(\Omega)}\leq R\}$, then $S(B_R)\subset S(L^{q'}(\Omega))\subset B_R$.
    
    Moreover, by the $L^{q'}$-continuity of $\Phi$, from the Proposition \ref{prop:continuous_dependence_v}, the solution map $S:L^{q'}(\Omega)\mapsto L^{q'}(\Omega)$ is completely continuous, by the compactness of $\Lambda^{s,p}_0(\Omega)\Subset L^{q'}(\Omega)$. Therefore, by Schauder's fixed point theorem, there exists a function $u=S(u)$ solving the implicit problem \eqref{eq:implicit_problem}.
\end{proof}
\begin {corollary}
Let  $f\in L^{p'}(\Omega)$, $\lambda_{\pm}\in L^{p'}(\Omega)$ with $\lambda_\pm(x)\geq 0$ a.e. $x\in\Omega$ and a continuous application $\Phi:L^{p}(\Omega)\to L^{p}(\Omega)$.
For $0<s<1$, let the triple $(u_s, \chi^s_{{}_{>}}[\Phi(u_s)], \chi^s_{{}_{< }}[\Phi(u_s)])\in \Lambda^{s,p}_0(\Omega)\times L^\infty(\Omega)\times L^\infty(\Omega)$ be a solution given in Theorem \ref{thm:existence_theorem_qvi_two_phase}. Then, for any $0\leq t<1$, there exists a subsequence $s\to 1$, 
    \begin{equation}\label{eq:convergence_u_s}u_s\to u \mbox{ in } \Lambda^{t,p}_0(\Omega),\quad D^su_s\to Du \mbox{ in } L^p(\R^d;\R^d),
    \end{equation}
    and
    \begin{equation}\label{eq:convergence_quasi_characteristics}
        \chi^s_{{}_{>}}[\Phi(u_s)]\overset{\ast}{\rightharpoonup}\chi^1_{{}_{>}}[\Phi(u)] \quad\mbox{ and }\quad \chi^s_{{}_{<}}[\Phi(u_s)]\overset{\ast}{\rightharpoonup}\chi^1_{{}_{<}}[\Phi(u)] \quad\mbox{ in } L^\infty(\Omega), 
    \end{equation}
    where $(u,\chi^1_{{}_{>}}[\Phi(u)], \chi^1_{{}_{<}}[\Phi(u)])\in W^{1,p}_0(\Omega)\times L^\infty(\Omega)\times L^\infty(\Omega)$ is a solution to \eqref{eq:implicit_problem} with $s=1$, i.e. with the $p$-Laplacian.
    
    In addition, if we assume also that \text{meas}$(\{u=\Phi(u)\})=0$, then
    \begin{equation*}
        \chi^s_{{}_{>}}[\Phi(u_s)]\to\chi_{\{u>\Phi(u)\}} \mbox{ and } \chi^s_{{}_{<}}[\Phi(u_s)]\to\chi_{\{u<\Phi(u)\}} \mbox{ strongly in } L^r(\Omega), r<\infty.
    \end{equation*}
\end{corollary}
\begin{proof}
Since the $u_s$ clearly satisfy the estimate \eqref{eq:estimate_for_frac_grad_regularised_solution}, by compactness we can extract a subsequence $s\to 1$, such that $u_s\to u$ in $\Lambda^{t,p}_0(\Omega)$, for some $u\in W^{1,p}_0(\Omega)$. Then $\Phi(u_s)\to \Phi(u)$ and the conclusions follow as in in Theorem \ref{thm:convergence_measure_fro_frac_to_cla}.
\end{proof}
\begin{corollary}[Unstable problem]
    Let $f\in L^{p^\#_s}(\Omega)$, $v\in L^{q'}(\Omega)$ and $\lambda_{\pm}\in L^q(\Omega)$ with $\lambda_\pm\geq 0$ a.e. $x\in\Omega$. There exist a triple $(u, \Tilde{\chi}_{{}_{>}}[v], \Tilde{\chi}_{{}_{<}}[v])\in \Lambda^{s,p}_0(\Omega)\times L^\infty(\Omega)\times L^\infty(\Omega)$ satisfying
    \begin{equation*}
        \begin{cases}
            -\Delta^s_p u+\lambda_+\Tilde{\chi}_{{}_{<}}[v]-\lambda_-\Tilde{\chi}_{{}_{>}}[v]=f &\mbox { in } \Omega\\
            u=0 &\mbox{ on } \R^d\setminus\Omega.
        \end{cases}
    \end{equation*}
    with
    \begin{equation*}
        \chi_{\{u>v\}}\leq \Tilde{\chi}_{{}_{>}}[v]\leq \chi_{\{u\geq v\}}, \quad{ and } \quad \chi_{\{u<v\}}\leq \Tilde{\chi}_{{}_{<}}[v]\leq \chi_{\{u\leq v\}}.
    \end{equation*}
\end{corollary}
\begin{proof}
    Simple application of Theorem \ref{thm:existence_theorem_qvi_two_phase} with $\Phi:w\to 2w-v$.
\end{proof}
Several examples can be provided for the interaction of the solution with the definition of the inter-phases. We give a few examples for the choice of the operator $\Phi$.
\begin{remark}
    We can consider $\Phi:L^{p}(\Omega)\to L^{p}(\Omega)$ to be a Nemytskii operator defined by 
    \begin{equation}
        \Phi(u)(x)=\phi(x,u(x)) 
    \end{equation}with a Carathéodory function $\phi:\Omega\times \R \to \R$ with at most affine growth, that is
    \begin{equation}\label{eq:growth_hypothesis_for nemyntskii}
        \phi(x,w)\leq \phi_0(x)+C_1|w|, \quad \mbox {for}\quad x\in\Omega, w\in\R,
    \end{equation}
    with $0\leq\phi_0\in L^{p}(\Omega)$ and $C_1\geq 0$.
    
    Another particularly interesting example is the case in which $\Phi$ is given by a Uryson operator,
    \begin{equation}
        \Phi(u)(x)=\int_\Omega{\tau(x, y, u(y))}\,dy,\quad \mbox{ for a.e. } x\in\Omega,
    \end{equation}
    where $\tau(x,y,w):\Omega\times \Omega\times\R\to\R$ is a Carathéodory function, i.e., continuous in $w$ and measurable for a.e. $(x,y)$, and satisfying the inequality
    \begin{equation*}
        |\tau(x,y,w)|\leq \phi(x,y)+C|w|, \mbox{ for a.e. } (x,y)\in\Omega\times\Omega \mbox{ and for all } r\in \R
    \end{equation*}
    with $0\leq \phi\in L^p(\Omega\times\Omega)=L^p(\Omega;L^p(\Omega))$ and $C>0$. In fact, to prove that $u\mapsto \Phi(u)=v$ is continuous, as an operator in $L^p(\Omega)$, we may consider any sequence of functions $w_n\in L^p(\Omega)$ converging to $w\in L^p(\Omega)$ as $n\to\infty$. We have $\tau(x,y, w_n(y))\to \tau(x,y,w(y))$ in measure as $n\to\infty$, for a.e. $(x,y)\in\Omega\times \Omega$, because $\tau$ is a Carathéodory function. Moreover, for any measurable subset $\mathcal{C}\subset \Omega$ we have
    \begin{equation*}
        \lim_{|\mathcal{C}|\to 0}{\int_{\mathcal{C}}{|\tau(x, y, w_n(y))|^p}\,dy}\leq \lim_{|\mathcal{C}|\to 0}2^p{\int_{\mathcal{C}}{[|\phi(x, y)|^p+C|w_n|^p]}\,dy}=0 
    \end{equation*}
    uniformly with respect to $n$, because of Vitalli's theorem applied to the sequence $w_n$ that converges to $w$ in $L^p(\Omega)$. Using again Vitalli's theorem, we obtain that $\tau(x, \cdot, w_n(\cdot))\to \tau(x, \cdot, w(\cdot))$ in $L^p(\Omega)$ for a.e. $x\in\Omega$.
    The continuity of $\Phi$ then follows from Lebesgue's dominated convergence theorem.
\end{remark}

\begin{remark}
 A quasi-variational two-phase obstacle-type problem can be given formally in terms of the coupled system of equations
    \begin{equation}\label{eq:auxiliary_problem_for_qvi}
        \begin{cases}
            -\Delta^s_p u+\lambda_+\chi_{\{u>v\}}-\lambda_-\chi_{\{u<v\}}=f &\mbox { in } \Omega\\
            -\Delta^{t}_p v=T(u) & \mbox{in }\Omega,\\
            u=v=0 &\mbox{on } \R^d\setminus \Omega,
        \end{cases}
    \end{equation}
    where $0<s\leq t\leq 1$ and $T$ is a suitable operator, for instance, the truncation defined by
    \begin{equation}\label{eq:truncation}
        T(w)(x)= max\{g_-(x), min[w(x), g_+(x)]\}\quad a.e.\,\,
        x\in\Omega,
    \end{equation}
    with $g_\pm\in L^\infty(\Omega)$, $g_-\leq g_+$, and $f,\lambda_{\pm}\in L^{\infty}(\Omega)$, for simplicity. Then, if $v=S(u)\in\Lambda^{t,p}_0(\Omega) $ is the solution of the Dirichlet problem for the $-\Delta^t_p$ with $T(u)\in L^\infty(\Omega)$ as given data, clearly, $u\mapsto \Phi(u)=v$ is a completely continuous operator in $L^p(\Omega)$. Consequently, Theorem \ref{thm:existence_theorem_qvi_two_phase} guarantees the existence of  $(u, \chi_{{}_{>}}[\Phi(u)], \chi_{{}_{< }}[\Phi(u)])\in \Lambda^{s,p}_0(\Omega)\times L^\infty(\Omega)\times L^\infty(\Omega)$,  which is a weak solution of the system \eqref{eq:auxiliary_problem_for_qvi} with $\Phi(u)=v$.

    Consider now the coupled system \eqref{eq:auxiliary_problem_for_qvi} with the Laplacian, i.e. with
    $s=t=1$ and $p=2$. Then, both $u,v\in W^{2,r}_\text{loc}(\Omega)$, for all $r<\infty $, by the linear theory of regularity to the Dirichlet problem. Then, as in Remark \ref{nondegenerate},  if, in addition, we assume 
    \begin{equation*}
        -g_+(x) +\lambda_+(x)<f(x)\quad\mbox{ or }\quad -g_-(x)-\lambda_-(x)>f(x) \quad \mbox{ a.e. in } \Omega.
    \end{equation*}
   the nondegeneracy property 
    \begin{equation*}
        \mathcal{L}^d(\Xi^1_{\Phi(u)})=\text{meas}\{u=\Phi(u)\}=0.
    \end{equation*}
    holds, since \eqref{eq:sufficient_condition_nondegeneracy_with_translation} is satisfied. Consequently, the quasi-characteristic functions $\chi_{{}_{>}}[\Phi(u)]$ and $\chi_{{}_{< }}[\Phi(u)])$ are in fact characteristic functions and the pair $(u,v)$ solves a.e. the nonlinear system \eqref{eq:auxiliary_problem_for_qvi}.
\end{remark}


\end{document}